\documentclass[reqno]{amsproc}
\usepackage[utf8]{inputenc}
\usepackage{amssymb}
\usepackage{euscript}
\usepackage{tikz}
\usepackage{extarrows}
\usepackage{mathrsfs}
\usepackage{a4wide}
\usepackage{setspace}

\makeatletter
\@namedef{subjclassname@2010}{%
\textup{2010} Mathematics Subject Classification}
\makeatother

\newtheorem{thm}{Theorem}
\newtheorem{cor}[thm]{Corollary}
\newtheorem{lem}[thm]{Lemma}
\newtheorem{pro}[thm]{Proposition}

\theoremstyle{remark}
\newtheorem{rem}[thm]{Remark}

\newtheorem*{opq}{Problem}

\theoremstyle{definition}
\newtheorem{exa}[thm]{Example}

\newcommand*{\card}[1]{\mathrm{card}(#1)}

\DeclareMathOperator{\D}{d\hspace{-0.25ex}}

\DeclareMathOperator{\dzii}{{\mathsf{Chi}}}

\DeclareMathOperator{\koo}{{\mathsf{root}}}

\DeclareMathOperator{\paa}{{\mathsf{par}}}

\DeclareMathOperator{\M}{m}

\newcommand*{\ascr}{\mathscr A}

\newcommand*{\borel}[1]{{\mathfrak B}(#1)}

\newcommand*{\cbb}{\mathbb C}

\newcommand*{\cfw}{C_{\phi,w}}
\newcommand*{\cfwa}{C_{\phi,w_\alpha}}

\newcommand*{\esf}{\mathsf{E}}

\newcommand*{\efw}{\mathsf{E}_{\phi,w}}

\newcommand*{\psf}{\mathsf{P}}
\newcommand*{\pfw}{\mathsf{P}_{\phi,w}}

\newcommand*{\dz}[1]{{\EuScript D}(#1)}
\newcommand*{\dzi}[1]{\dzii(#1)}

\newcommand*{\Ge}{\geqslant}

\newcommand*{\hh}{\mathcal H}

\newcommand*{\hsf}{{\mathsf h}}

\newcommand*{\hfw}{{\mathsf h}_{\phi,w}}

\newcommand*{\is}[2]{\langle#1,#2\rangle}
\newcommand*{\jd}[1]{\EuScript N(#1)}

\newcommand*{\lambdab}{{\boldsymbol\lambda}}

\newcommand*{\Le}{\leqslant}

\newcommand*{\nbb}{\mathbb N}

\newcommand*{\ogr}[1]{\boldsymbol B(#1)}

\newcommand*{\pa}[1]{\paa(#1)}

\newcommand*{\rbb}{\mathbb R}
\newcommand*{\rbop}{{\overline{\rbb}_+}}
\newcommand*{\slam}{S_{\boldsymbol \lambda}}
\newcommand*{\smalloplus}{\raise0pt\hbox{$\scriptscriptstyle \oplus$}}

\newcommand*{\tcal}{{\mathscr T}}

\newcommand*{\zbb}{\mathbb Z}

\begin{document}
\setstretch{1.1}
\title[Weakly centered  wco's]{Weakly centered weighted composition operators in $L^2$-spaces}

\author[P.\ Budzy\'{n}ski]{Piotr Budzy\'{n}ski}
\address{Katedra Zastosowa\'{n} Matematyki, Uniwersytet Rolniczy w Krakowie, ul.\ Balicka 253c, 30-198 Kra\-k\'ow, Poland}
\email{piotr.budzynski@urk.edu.pl}

\date{\today}

\begin{abstract}
Weakly centered and spectrally weakly cenetered weighted composition operators in $L^2$-spaces are characterized. Criteria for existence of invariant subspaces are given. Additional results and examples are supplied.
\end{abstract}
\maketitle

\section{Introduction}
This paper is concerned with operators in Hilbert spaces satisfying a commutative relation 
\begin{align}\label{fatboy}
TT^*T^*T=T^*TTT^*
\end{align}
(or its spectral analogue). We call them {\em weakly centered}. They have been formally introduced by Campbell (under a name of binormal operators). He characterized them and showed that {\em weakly centered hyponormal operators have invariant subspaces} (see \cite{1972-pams-campbell, c-pjm-1974}). Embry considered operators satisfying \eqref{fatboy} even earlier in relation to normality. In fact, she showed among other things that a {\em weakly centered operator with a positive real part or $0$ outside of its numerical range is normal} (see \cite{1966-pjm-embry, 1970-pjm-embry}). Later, these operators were subject to a research related to similarity to contractions carried by Paulsen, Pearcy, and Petrovi{\'c} (see \cite{1992-jfa-petrovic, 1995-jfa-paulsen-pearcy-petrovic}). All the above results spark our interest in weakly centered operators.

We restrict our attention to a particular class of Hilbert space operators -- the class of weighted composition operators in $L^2$-spaces, which act according to the following formula
$$f\longmapsto w\cdot (f\circ \phi).$$
The class includes multiplication operators, composition operators, partial (weighted and unweighted) composition operators and classical weighted shifts. The last ones have a very interesting generalization -- the so called weighted shifts on directed trees (see \cite{2012-mams-j-j-s}). As it turns out, all the weighted shifts on directed trees with cardinality less or equal to $\aleph_0$ are weighted composition operators too. Studies over weighted composition operators are rooted in the classical Banach-Stone theorem on one side and in ergodic theory on the other. The theory is lively and has been strengthened with new interesting findings in recent years (see eg., \cite{2017-bims-azi-jab-jaf, 2022-mjm-azimi-jabbarzadeh, 2024-rim-budzynski, 2015-jfa-budzynski-jablonski-jung-stochel, 2017-aim-budzynski-jablonski-jung-stochel, 2018-lnim-budzynski-jablonski-jung-stochel, 2012-asm-pozzi, 2024-laa-stochel-stochel})

Motivated by results of Campbell we provide a characterization of weakly centered weighted composition operators (see Theorem \ref{wcent01}) and a criterion for existence of invariant subspaces (see Proposition \ref{suka01}). The results, as per usual in the context of weighted composition operators, are written in terms of the associated Radon-Nikodym derivative and conditional expectation. The proofs rely on classical measure theoretic methods but we use also the Aluthge transform, an operator theoretic tool that was introduced to study $p$-hyponormal and $\log$-hyponormal operators (see \cite{1990-ieot-aluthge}), and later led to numerous substantial results on the transform itself and its relation to the associated operator (see eg., \cite{2005-ieot-cho-jung-lee, 2003-pjm-foias-jung-ko-pearcy, 2003-ieot-jung-ko-pearcy, 2012-ieot-lee-lee-yoon}). 

Later in the paper we make a venture into the realm of unbounded operators. The condition characterizing weakly centered weighted composition operators make perfect sense whenever the operator in question is densely defined. A problem arising naturally is whether the conditions imply some sort of commutativity of $\cfw^*\cfw$ and $\cfw\cfw^*$ in the unbounded case. This leads us to introducing {\em spectrally weakly centered} operators. We show that spectrally weakly centered weighted composition operators are characterized essentially in the same way as bounded ones (Theorem \ref{gamon01}). We also provide an unbounded weighted composition operator analogue of the Campbell's (Theorem \ref{smutek01}).
%
\section{Preliminaries}
We write $\zbb$, $\rbb$, and $\cbb$ for the sets of integers, real numbers, and complex numbers, respectively. By $\nbb$, $\zbb_+$, and $\rbb_+$ we denote the sets of positive integers, nonnegative integers, and nonegative real numbers, respectively. $\rbop$ stands for $\rbb_+ \cup \{\infty\}$. $\borel{\cbb}$ denotes the $\sigma$-algebra of Borel subsetes of $\cbb$. $\ogr{\hh}$ stands for a $\mathcal{C}^*$-algebra of bounded operators on a (complex) Hilbert space $\hh$. If $A$ is an operator in $\hh$ (possibly unbounded), then $\dz{A}$, $\jd{A}$, $\mathcal{R}(A)$, and $A^*$ stand for the domain, the kernel, the range, and the adjoint of $A$, respectively. A densely defined operator $T$ in $\hh$ is {\em hyponormal} if $\dz{T}\subseteq \dz{T^*}$ and $\|T^*f\| \leqslant \|Tf\|$ for all $f\in\dz{T}$.

Throughout the paper we assume that $(X,\ascr, \mu)$ is a $\sigma$-finite measure space, $\phi$ is an $\ascr$-measurable {\em transformation} of $X$ (i.e., an $\ascr$-measurable mapping $\phi\colon X\to X$) and $w\colon X\to\cbb$ is $\ascr$-measurable. 

Let $\mu_w$ and $\mu_w\circ\phi^{-1}$ be measures on $\ascr$ defined by $\mu_w(\sigma)=\int_\sigma|w|^2\D\mu$ and $\mu_w\circ \phi^{-1}(\sigma)=\mu_w\big(\phi^{-1}(\sigma)\big)$ for $\sigma\in\ascr$. Assume that $\mu_w\circ\phi^{-1}$ is absolutely continuous with respect to  $\mu$. Then the operator 
\begin{align*}
\cfw \colon L^2(\mu) \supseteq \dz{\cfw} \to L^2(\mu)    
\end{align*}
given by
\begin{align*}
\dz{\cfw} = \{f \in L^2(\mu) \colon w \cdot (f\circ \phi) \in L^2(\mu)\},\quad
\cfw f  = w \cdot (f\circ \phi), \quad f \in \dz{\cfw},
\end{align*}
is well defined (see \cite[Proposition 7]{2018-lnim-budzynski-jablonski-jung-stochel}); here, as usual, $L^2(\mu)$ stands for the complex Hilbert space of all square $\mu$-integrable $\ascr$-measurable complex functions on $X$ (with the standard inner product). $\cfw$ is called a {\em weighted composition operator} (abbreviated to {\em wco} throughout the paper). By the Radon-Nikodym theorem (cf.\ \cite[Theorem 2.2.1]{ash}), there exists a unique (up to a set of $\mu$-measure zero) $\ascr$-measurable function $\hfw\colon X \to \rbop$ such that
 \begin{align*}
\mu_w \circ \phi^{-1}(\varDelta) = \int_{\varDelta} \hfw \D \mu, \quad
\varDelta \in \ascr.
\end{align*}
As it follows from \cite[Theorem 1.6.12]{ash} and \cite[Theorem 1.29]{rud}, for every $\ascr$-measurable function $f \colon X \to \rbop$ (or for every $\ascr$-measurable function $f\colon X \to \cbb$ such that $f\circ \phi \in L^1(\mu_w)$), we have
\begin{align*}
\int_X f \circ \phi \D\mu_w = \int_X f \, \hfw \D \mu.
\end{align*}
In particular, this implies that $\dz{\cfw}=L^2\big((1+\hfw)\D\mu\big)$. Recall that $\cfw$ is a bounded operator defined on the whole of $L^2(\mu)$ if and only if $\hfw$ is $\mu$-essentially bounded and, if this is the case, $\|\cfw\|^2=\|\hfw\|_{L^\infty(\mu)}$ (see \cite[Proposition 8]{2018-lnim-budzynski-jablonski-jung-stochel}). 

Recall that for every closed densely defined operator $A$ in a Hilbert space $\hh$ there exists a (unique) partial isometry $U$ such that $\jd{U}=\jd{A}$ and $A=U|A|$, where $|A|$ is the square root of $A^*A$; the operator $U$ is called the {\em phase} of $A$ and $|A|$ is the {\em modulus} of $A$. In view of \cite[Theorem 18]{2018-lnim-budzynski-jablonski-jung-stochel}, assuming $\cfw$ is densely defined, the phase of $\cfw$ is equal to $C_{\phi,\widetilde w}$ with $\widetilde w\colon C\to \cbb$ given by
\begin{align}\label{terefere-1}
\widetilde w= \frac{w}{\sqrt{\hfw\circ\phi}},
\end{align}
and the modulus of $\cfw$ is equal to $M_{\sqrt{\hfw}}$. Here and later on, for any given an $\ascr$-measurable function $g\colon X\to \cbb$, $M_g$ denotes the operator of multiplication by $g$ acting in $L^2(\mu)$.

Assume that $\cfw$ is densely defined. Then, for a given $\ascr$-measurable function $f\colon X\to \rbop$ (or $f\colon X\to \cbb$ such that $f\in L^p(\mu_w)$ with some $1\leqslant p<\infty$) one may consider $\efw(f)$, the conditional expectation of $f$ with respect to the $\sigma$-algebra $\phi^{-1}(\ascr)$ and the measure $\mu_w$ (see \cite[Section 2.4]{2018-lnim-budzynski-jablonski-jung-stochel}): $\efw(f)$ is a unique (up to set of $\mu_w$-measure) $\phi^{-1}(\ascr)$-measurable function such that \begin{align*}
\int_\varDelta f \D\mu_w=\int_\varDelta \efw(f)\D\mu_w,\quad \varDelta\in \phi^{-1}(\ascr).
\end{align*} 
Moreover, there is a unique (up to sets of $\mu$-measure zero) $\ascr$-measurable function $g$ on $X$ such that $g=0$ a.e. $[\mu]$ on $\{\hfw=0\}$ and $\efw(f)=g\circ\phi$ a.e. $[\mu_w]$. We will denote this function by $\efw(f)\circ\phi^{-1}$. Recall that $\efw$ can be regarded as a linear contraction on $L^p(\mu_w)$, $p\in[1,\infty]$, which leaves invariant the convex cone $L^p_+(\mu_w)$ of all $\rbb_+$-valued members of $L^2(\mu_w)$; moreover, $\efw$ is an orthogonal projection on $L^2(\mu_w)$. 

Considering $w\equiv 1$ leads to a large subclass of weighted composition operators, the so-called composition operators. Namely, assuming that $\phi\colon X\to X$ is $\ascr$-measurable and satisfies $\mu\circ\phi^{-1}\ll\mu$, the operator $C_{\phi, 1}$ is well defined. We call it the {\em composition operator} induced by $\phi$; for brevity we use the notation $C_\phi:=C_{\phi,1}$. The corresponding Radon-Nikodym derivative $\mathsf{h}_{\phi,1}$ and the conditional expectation $\esf_{\phi,1}$ are denoted by $\mathsf{h}_\phi$ and $\esf_\phi$, respectively. We caution the reader that $C_\phi$ might not be well defined even if $\cfw$ is an isometry. In particular, if this is the case, $\hsf_\phi$ and $\esf_{\phi}$ do not exist whereas $\hfw$ and $\efw$ do (see \cite[pg. 71]{2018-lnim-budzynski-jablonski-jung-stochel}).

Let $\alpha\in(0,1]$. Suppose $A$ is closed densely defined operator in $\hh$ and $A = U|A|$ is its polar decomposition. The $\alpha$-Aluthge transform of $A$ is defined as
\begin{align*}
\Delta_\alpha(A) = |A|^\alpha U|A|^{1-\alpha}.
\end{align*}
Recall that $\Delta_\alpha(A)$ may have a trivial domain. This can happen for $A=\cfw$ (see \cite{2020-mn-benhida-budzynski-trepkowski, 2015-jmaa-trepkowski}).

The Aluthge transform of wco's has been described in \cite[Theorem 3.1]{2020-mn-benhida-budzynski-trepkowski}. Recall that whenever $\alpha\in(0,1]$ and $\cfw$ is densely defined, we have
\begin{align*}
\dz{\Delta_\alpha(\cfw)}=L^2\Big(\big(1+\hfw^{1-\alpha}+\efw(\hfw^\alpha)\circ\phi^{-1}\hfw^{1-\alpha}\big)\D\mu\Big)
\end{align*}
and
\begin{align*}
\Delta_\alpha(\cfw)\subseteq \cfwa,
\end{align*}
where $w_\alpha\colon X \to \cbb$ is given by
\begin{align}\label{terefere1}
w_\alpha=w\cdot \bigg(\frac{\hfw}{\hfw\circ\phi}\bigg)^{\alpha/2}\quad \text{a.e.\ $[\mu]$}.
\end{align}
The operator $\Delta_\alpha(\cfw)$ is closable and $\overline{\Delta_\alpha(\cfw)}=\cfwa$. In particular, if $\cfw$ is a bounded operator on ${L^2(\mu)}$, we have 
\begin{align}\label{auto01}
\Delta_\alpha(\cfw)=\cfwa    
\end{align}
Assuming additionally that $C_\phi$ is a well defined operator on $L^2(\mu)$ (this is not automatically satisfied as shown in \cite[Example 102]{2018-lnim-budzynski-jablonski-jung-stochel}), we deduce from \cite[Theorem 110]{2018-lnim-budzynski-jablonski-jung-stochel} that $\Delta_\alpha(\cfw)=M_{w_\alpha}C_\phi$.
\section{The bounded case}
Weakly centered wco's can be characterized in terms of the associated Radon-Nikodym derivative and the (partial) isometric part of the polar decomposition of the Aluthge transform. Some additional notation will be helpful: we set
\begin{align*}
\widetilde w_{\alpha}=\frac{w_\alpha}{\sqrt{\hsf_{\phi,w_\alpha}\circ\phi}},
\end{align*}
for $\alpha>0$ and any $w\colon X\to \mathbb{C}$ such that the above make sense (cf. \eqref{terefere-1} and \eqref{terefere1}).
\begin{thm}\label{wcent01}
Assume $\cfw\in\ogr{L^2(\mu)}$. Then the following conditions are equivalent:
\begin{itemize}
\item[(i)] $\cfw$ is weakly centered,
\item[(ii)] $C_{\phi, \chi_{\{\hfw>0\}} \widetilde w}=C_{\phi, \widetilde w_{1/2}}$,
\item[(iii)] $\hfw=\efw \big(\hfw\big)$ a.e. $[\mu_w]$.
\end{itemize}
\end{thm}
\begin{proof}
(i) $\Rightarrow$ (ii) According to \cite[Proposition 3.9]{2004-jot-ito-yamazaki-yanagida}, if $\cfw$ is weakly centered then the phase $V_{1/2}$ of $\Delta_{1/2}(\cfw)=V_{1/2} |\Delta_{1/2}(\cfw)|$ coincides with $U^*UU$, where $U$ is the phase of $\cfw=U|\cfw|$. In view of \eqref{auto01} we have $\Delta_{1/2}(\cfw)=C_{\phi, w_{1/2}}$ and so 
\begin{align}\label{cisza02}
V_{1/2}=C_{\phi, \widetilde w_{1/2}}.    
\end{align}
Since, $U=C_{\phi,\widetilde w}$, we see that $U^*U=M_{{\hsf_{\phi,\widetilde w}}}=M_{\chi_{\{\hfw>0\}}}$ by \cite[Theorem 18 and pg. 29]{2018-lnim-budzynski-jablonski-jung-stochel}. Thus 
\begin{align}\label{cisza03}
U^*UU=M_{\chi_{\{\hfw>0\}}} C_{\phi,\widetilde w}.    
\end{align}
From $|\chi_{\{\hfw>0\}}\widetilde w|^2\leqslant |\widetilde w|^2$ a.e. $[\mu]$ we deduce that $C_{\phi, \chi_{\{\hfw>0\}} \widetilde w}$ is well defined. Clearly, for any $f\in L^2(\mu)$, $M_{\chi_{\{\hfw>0\}}} C_{\phi,\widetilde w}f =M_{\chi_{\{\hfw>0\}}} C_{\phi,\widetilde w}f$. This, \eqref{cisza02} and \eqref{cisza03} imply (ii) whenever (i) is satisfied.

(ii) $\Rightarrow$ (iii) By \cite[Lemma 3.1]{2020-mn-benhida-budzynski-trepkowski} we get 
\begin{align*}
\hsf_{\phi, w_{1/2}}\circ \phi = \efw\big(\sqrt{\hfw}\big)\sqrt{\hfw\circ\phi} \text{ a.e. $[\mu_w]$}.
\end{align*}
Assuming $M_{\chi_{\{\hfw>0\}}}C_{\phi, \widetilde w}=C_{\phi, \widetilde w_{1/2}}$ we get 
\begin{align*}
\chi_{\{\hfw>0\}}\frac{w}{\sqrt{\hfw\circ \phi}} (f\circ \phi) = \frac{w}{\sqrt{\hfw\circ\phi}} \frac{\sqrt[4]{\hfw}}{\sqrt{\efw\big(\sqrt{\hfw}\big)}} (f\circ\phi),\quad f\in L^2(\mu).
\end{align*}
This in turn can only be satisfied if and only if $\frac{\sqrt{\hfw}}{\efw(\sqrt{\hfw})}=1$ a.e. $[\mu_w]$ on $\{\hfw>0\}$ or, equivalently, $\hfw>0$ a.e. $[\mu_w]$ and $\sqrt{\hfw}=\efw\big(\sqrt{\hfw}\big)$ a.e. $[\mu_w]$ (use \cite[Lemma 6 and Proposition 8(v)]{2018-lnim-budzynski-jablonski-jung-stochel}). The latter is equivalent to $\hfw=\efw\big(\hfw\big)$ a.e. $[\mu_w]$ as $\psi(t)=\sqrt{t}$ is a bijective transformation of $(0,+\infty)$. 

(iii) $\Rightarrow$ (i) By \cite[Theorem 18 and (2.10)]{2018-lnim-budzynski-jablonski-jung-stochel} we have
\begin{align*}
|\cfw| U |\cfw| U^*f=&M_{\sqrt{\hfw}}C_{\phi,\widetilde w}M_{\sqrt{\hfw}}\big( \sqrt{\hfw}\cdot \efw(f_w)\circ\phi^{-1}\big)\\
=&\sqrt{\hfw \hfw\circ \phi}\cdot w\cdot \efw(f_w),\quad f\in L^2(\mu).
\end{align*}
Assuming (iii), we have
\begin{align*}
U |\cfw| U^* |\cfw| f=&C_{\phi,\widetilde w}M_{\sqrt{\hfw}}\Big( \sqrt{\hfw}\cdot \efw\big((\sqrt{\hfw}f)_w\big)\circ\phi^{-1}\Big)\\
=&\widetilde w\cdot (\hfw\circ\phi)\cdot \big( \efw\big(\sqrt{\hfw}f_w\big)\big)\\
=&\sqrt{\hfw \hfw\circ \phi}\cdot  w\cdot 
\efw(f_w),\quad f\in L^2(\mu)
\end{align*}
Thus (iii) imply that $|\cfw|$ and $U|\cfw|U^*$ commute. This, by \cite[Theorem 1]{1972-pams-campbell}, is equivalent to (i). This completes the proof.
\end{proof}
\begin{rem}
Inspecting the proof of Theorem \ref{wcent01} one can derive other equivalent conditions for $\cfw$ to be weakly centered. Indeed, since $\psi(t)=t^\alpha$ is a bijective transformation of $(0,+\infty)$ for any $\alpha\in \rbb\setminus\{0\}$, the condition (iii) of Theorem \ref{wcent01} is equivalent to any of the following
\begin{enumerate}
\item[(iv)] $\hfw^\alpha=\efw \big(\hfw^\alpha\big)$ a.e. $[\mu_w]$ for some $\alpha\in\rbb\setminus\{0\}$,
\item[(v)] $\hfw^\alpha=\efw \big(\hfw^\alpha\big)$ a.e. $[\mu_w]$ for every $\alpha\in\rbb\setminus\{0\}$. 
\end{enumerate}
For any $\alpha\in(0,1]$, the phase $V_\alpha$ in the polar decomposition $\Delta_\alpha (\cfw)=V_\alpha |\Delta_\alpha (\cfw)|$ of the $\alpha$-Aluthge transform $\Delta_\alpha (\cfw)$ is equal to $C_{\phi, \widetilde w_\alpha}$, with 
\begin{align*}
\widetilde w_\alpha=\frac{w}{\sqrt{\hfw\circ\phi}}\frac{\sqrt{\hfw^{\alpha}}}{\sqrt{\efw\big(\hfw^\alpha\big)}},
\end{align*}
Using the above we deduce the equivalence of (v) to
\begin{enumerate}
\item[(vi)] $V_{\alpha_1}=V_{\alpha_2}$ a.e. $[\mu_w]$ for all $\alpha_1, \alpha_2\in(0,1]$. 
\end{enumerate}
\end{rem}
\begin{rem}\label{szaruga01}
Regarding the Aluthge transform and its use in the proof of Theorem \ref{wcent01} it is worth recalling that it was shown in \cite[Theorem 3.1]{2004-jot-ito-yamazaki-yanagida} that {\em an operator $T$ is weakly centred if and only if $\Delta_{1/2}(T)=U_T|\Delta_{1/2}(T)|$, where $U_T$ is the phase of $T$}. This however does not mean that $U_T$ is equal to the phase of $\Delta_{1/2}(T)$. In particular, the kernels of these two operators may differ. In another words, $\Delta_{1/2}(T)=U_T|\Delta_{1/2}(T)|$ in general is not the polar decomposition of $\Delta_{1/2}(T)$. Should this be true for $T=\cfw$, the Radon-Nikodym derivative $\hfw$ would satisfy ``$\hfw>0$ a.e. $[\mu_w]$'', which can be deduced inspecting the proof of Theorem \ref{wcent01}. 

The condition ``$\hfw>0$ a.e. $[\mu_w]$'' has been characterized in \cite[Proposition 12]{2018-lnim-budzynski-jablonski-jung-stochel}. For a densely defined $\cfw$ it is equivalent to a condition
``$\mu\big(\{\hfw=0\}\cap\{w\neq 0\}\big)=0$''. One can use this to show that ``$\hfw>0$ a.e. $[\mu_w]$'' does not follow from ``$\hfw=\efw(\hfw)$ a.e. $[\mu_w]$''. In fact, this can even be done even for composition operators. For this, let's consider $(X, \ascr,\mu)=(\nbb, 2^\nbb, \nu)$ with $\nu$ equal to the counting measure, and an injective $\phi\colon X\to X$ such that $\phi(X)\neq X$. By \cite[Proposition 79]{2018-lnim-budzynski-jablonski-jung-stochel}, $\hsf_\phi(x)=1$ for every $x\in \phi(X)$ and $\hsf_\phi(x)=0$ for every $x\in X\setminus\phi(X)$. Since $\hsf_\phi$ is constant on $\phi^{-1}(\{x\})$ for every $x\in\phi(X)$, $\hsf_\phi=\esf_\phi(\hsf_\phi)$. Clearly $\mu\big(\{\hsf_\phi=0\}\big)\neq0$.
\end{rem}
Three immediate corollaries follow.
\begin{cor}\label{debile01}
Let $C_\phi\in \ogr{L^2(\mu)}$. Then $C_\phi$ is weakly centered if and only if $\esf_{\phi}(\hsf_\phi)=\hsf_{\phi}$.
\end{cor}
\begin{cor}\label{sad03}
Let $\cfw\in \ogr{L^2(\mu)}$. If $\efw$ is the identity on $L^2(\mu_w)$, then $\cfw$ is weakly centered.
\end{cor}
\begin{cor}
Let $C_\phi\in \ogr{L^2(\mu)}$. If $\cfw$ is cohyponormal (i.e., $\cfw^*$ is hyponormal), then $\cfw$ is weakly centered.
\end{cor}
\begin{proof}
Use \cite[Theorem 60]{2018-lnim-budzynski-jablonski-jung-stochel}.
\end{proof}
The phase of a weakly centered wco has to be quasinormal.
\begin{pro}\label{black01}
Assume that $\cfw\in\ogr{L^2(\mu)}$ is weakly centered. Then $C_{\phi, \widetilde w}$ is quasinormal, i.e., $C_{\phi, \widetilde w}^*C_{\phi, \widetilde w}C_{\phi, \widetilde w}=C_{\phi, \widetilde w}C_{\phi, \widetilde w}^*C_{\phi, \widetilde w}$.
\end{pro}
\begin{proof}
In view of \cite[pg. 29]{2018-lnim-budzynski-jablonski-jung-stochel} and \cite[Lemma 6]{2018-lnim-budzynski-jablonski-jung-stochel}, we have $\hsf_{\phi,\widetilde w}=\chi_{\{\hfw>0\}}$ a.e. $[\mu]$ and 
\begin{align}\label{cisza01}
\hsf_{\phi,\widetilde w}\circ \phi=\chi_{\{\hfw\circ\phi>0\}}=1 \quad \text{a.e. $[\mu_w]$}.
\end{align}
Hence 
\begin{align*}
w_{\widetilde{1/2}}:=\widetilde w \sqrt[4]{\frac{\hsf_{\phi,\widetilde w}}{\hsf_{\phi,\widetilde w}\circ \phi}}=\chi_{\{\hfw>0\}} w\frac{1}{\sqrt{\hfw\circ\phi}}= \chi_{\{\hfw>0\}}\widetilde w \quad \text{a.e. $[\mu_w]$},
\end{align*}
and thus
\begin{align}\label{black02}
\Delta_{1/2}(C_{\phi, \widetilde w})f=C_{\phi ,w_{\widetilde{1/2}}}f=\chi_{\{\hfw>0\}} \widetilde w\, (f\circ\phi),\quad f\in L^2(\mu).
\end{align}
By Theorem \ref{wcent01}, $\sqrt{\hfw}=\efw\big(\sqrt{\hfw}\big)$ a.e. $[\mu_w]$. Therefore, the phase of $\Delta_{1/2}(\cfw)$ is given by
\begin{align}\label{black03}
C_{\phi, \widetilde w_{1/2}} f=\frac{w}{\sqrt{\hfw\circ\phi}} \frac{\sqrt[4]{\hfw}}{\sqrt{\efw\big(\sqrt{\hfw}\big)}} (f\circ\phi)=w(f\circ\phi),\quad f\in L^2(\mu).
\end{align}
According to \cite[Theorem 3.1]{2004-ieot-i-y-y}, $\Delta_{1/2}(\cfw)=\Delta_{1/2}(C_{\phi, \widetilde w})|\Delta_{1/2}(\cfw)|$ is the polar decomposition, yielding $\Delta_{1/2}(C_{\phi, \widetilde w})=C_{\phi,\widetilde w_{1/2}}$.
Hence, comparing \eqref{black02} and \eqref{black03}, we get
\begin{align}\label{black04}
\chi_{\{\hfw>0\}}=1\quad \text{a.e $[\mu_w]$}. 
\end{align}
Since $\mu_w$ and $\mu_{\widetilde w}$ are mutually absolutely continuous, \eqref{cisza01} and \eqref{black04} imply
\begin{align*}
\hsf_{\phi,\widetilde w}=\hsf_{\phi,\widetilde w}\circ \phi\quad \text{a.e. $[\mu_{\widetilde w}]$}.
\end{align*}
This and \cite[Theorem 20]{2018-lnim-budzynski-jablonski-jung-stochel} means that $C_{\phi, \widetilde w}$ is quasinormal.
\end{proof}
\begin{rem}\label{blackblack++}
Quasinormality of $C_{\phi,\widetilde w}$ does not imply that $\cfw$ is weakly centered -- any non weakly centered $\cfw$ such that $\hfw>0$ a.e. $[\mu]$ serves as a counter example and there is an abundance of such $\cfw$'s, (a concrete one is provided in Example \ref{blackblack}). However, if one assumes more, namely that $C_{\phi,\widetilde w}$ is unitary and $\cfw$ is a weighted shift on a directed tree, then $\cfw$ is centered (an essentially stronger property) as shown in Proposition \ref{orange}. In fact, this holds true regardless of whether the operator in question is a weighted shift on a directed tree or not, which was shown in \cite[Proposition 3.4]{1995-jfa-paulsen-pearcy-petrovic}, but in a context of weighted shifts one can obtain more specific information.
\end{rem}
\begin{rem}
It is worth pointing out that the quasinormality of $C_{\phi,\widetilde w}$ can also be deduced from the condition ``$\hfw>0$ a.e. $[\mu_w]$'' (mentioned in Remark \ref{szaruga01}) without assuming that $\cfw$ is weakly centered.
\end{rem}
Below we explicate a wide class of weakly centered composition and weighted composition operators, induced by linear transformations of $\rbb^n$.
\begin{exa}\label{sad04}
Let $n\in \nbb$. Let $X=\rbb^n$ and $\ascr=\borel{\rbb^n}$. Let $\rho(z) = \sum_{k=0}^\infty a_k z^k$, $z \in \cbb$, be an entire function such that $a_n$ is non-negative for every $k\in\zbb_+$ and $a_{k_0} > 0$ for some $k_0 \Ge 1$. Let $\mu=\mu^\rho$ be the $\sigma$-finite measure on $\ascr$ defined by
\begin{align*}
\mu^\rho(\sigma) = \int_\sigma \rho\big(\|x\|^2\big) \D\M_n(x),\quad \sigma\in\borel{\rbb^n},
\end{align*}
where $\M_n$ is the $n$-dimensional Lebesgue measure on $\rbb^n$ and $\|\cdot\|$ is a norm on $\rbb^n$ induced by an inner product. Finally, let $\phi\colon \rbb^n\to\rbb^n$ be invertible and linear. Then $\phi$ induces a composition operator $C_\phi$ in $L^2(\mu)$. The boundedness of $C_\phi$ has been fully characterized in \cite[Proposition 2.2]{sto-1990-hmj}: \emph{if $\rho$ is a polynomial, then $C_\phi$ is bounded; if $\rho$ is not a polynomial, then $C_\phi$ is bounded if and only if  $\|\phi^{-1}\|\Le 1$}. Therefore, if $\rho$ is polynomial and/or $\|\phi^{-1}\|\Le 1$, we have $C_\phi\in \ogr{L^2(\mu)}$. Since $\phi$ is an invertible transformation of $\rbb^n$, $\esf_\phi$ acts as an identity on $L^2(\mu)$ and thus, by Corollary \ref{sad03}, $C_\phi$ is weakly centered. Moreover, under the same assumptions on $\phi$ and $\rho$, if $w\colon \rbb^n\to \cbb$ is essentially bounded, $\cfw\in\ogr{L^2(\mu)}$ (use \cite[Proposition 111]{2018-lnim-budzynski-jablonski-jung-stochel}) and, by Corollary \ref{sad03} again, $\cfw$ is weakly centered too.
\end{exa}
In the reminder of the section we focus mostly on weighted shifts on directed trees. Necessary background information on these operators in the bounded case is provided below (for more we refer the interested reader to an excellent monograph \cite{2012-mams-j-j-s}). 

Let $\tcal=(V,E)$ be a directed tree ($V$ and $E$ stand for the sets of vertices and edges of $\tcal$, respectively). Set $\dzi u = \{v\in V\colon (u,v)\in E\}$ for $u \in V$. Denote by $\paa$ the partial function from $V$ to $V$ which assigns to a vertex $u\in V$ its parent $\pa{u}$ (i.e.\ a unique $v \in V$ such that $(v,u)\in E$). A vertex $u \in V$ is called a {\em root} of $\tcal$ if $u$ has no parent. A root is unique (provided it exists); we denote it by $\koo$. Set $V^\circ=V \setminus \{\koo\}$ if $\tcal$ has a root and $V^\circ=V$ otherwise. We say that $u \in V$ is a {\em branching vertex} of $V$, and write $u \in V_{\prec}$, if $\dzi{u}$ consists of at least two vertices. 

Assume $\lambdab=\{\lambda_v\}_{v \in V^{\circ}} \subseteq \cbb$ satisfies $\sup_{v\in V}\sum_{u\in\dzii{v}}|\lambda_u|^2<\infty$. Then the following formula
\begin{align*}
(\slam f)(v)=
   \begin{cases}
\lambda_v \cdot f\big(\pa v\big) & \text{ if } v\in V^\circ,
   \\
0 & \text{ if } v=\koo,
   \end{cases}
\end{align*}
defines a bounded operator on $\ell^2(V)$ (as usual, $\ell^2(V)$ is the Hilbert space of square summable complex functions on $V$ with standard inner product). We call it a {\em weighted shift on} $\tcal$ with weights $\lambdab$.
\begin{pro}\label{suka01}
Let $\slam\in\ogr{\ell^2(V)}$. Then $\slam$ is weakly centered if and only if for every $v\in V$ and every $u_1, u_2\in \dzi{v}$ such that $\lambda_{u_1}\lambda_{u_2}\neq 0$ the equality
$$\sum_{y\in\dzi{u_1}}|\lambda_{y}|^2=\sum_{y\in\dzi{u_2}}|\lambda_{y}|^2$$
holds.
\end{pro}
\begin{proof}
Set $X=V$ and $\ascr=2^V$. Let $\mu$ be the counting measure on $X$. Let $\phi\colon X\to X$ be any $\ascr$-measurable extension of the $\paa\colon V\to V$ partial function. Define $w\colon X\to \cbb$ by
\begin{align*}
w(x)=
   \begin{cases}
\lambda_x & \text{ if } x\in V^\circ,
   \\
0 & \text{ if } v=\koo.
   \end{cases}
\end{align*}
It is easy to see that $\slam$ and $\cfw$ coincide and thus, in view of Theorem \ref{wcent01},  $\slam$ is weakly centered if and only if $\efw(\hfw)=\hfw$ a.e. $[\mu_w]$, in another words $\hfw$ is $\phi^{-1}(\ascr)$-measurable $\mu_w$-almost everywhere. Since $(X,\ascr, \mu)$ is a discrete measure space, any $f\colon X\to \rbb_+$ is $\phi^{-1}(\ascr)$-measurable if and only if $f$ is constant on sets $\phi^{-1}(x)$, $x\in X$. By \cite[Proposition 79]{2018-lnim-budzynski-jablonski-jung-stochel}, we have
\begin{align*}
\hfw(x)={\sum_{y\in\phi^{-1}(\{x\})}|w(y)|^2}=\sum_{y\in\dzi{x}}|\lambda_y|^2,\quad x\in X.
\end{align*}
Since we care only for $x$'s that are atoms of $\mu_w$, we compare values of $\hfw$ at vertices with positive weights attached. This yields the claim.
\end{proof}
Recall that all classical weighted shifts are weakly centered (in fact, they are centered). With the above result at hand one can show that the simplest possible directed tree not related to classical weighted shifts -- a rootless directed tree with one branching vertex of valency 2 - admits a non weakly centered weighted shift.
\begin{exa}\label{blackblack}
Let $\tcal = (V,{E})$ be the directed tree with
\begin{align*}
    V = \big\{-k\colon k\in\zbb_+\big\} \cup \big\{ (i,j) \colon  i\in\{1,2\},\ j \in \nbb \big\}
\end{align*}
and $E$ as in the Figure \ref{drzewot201}. (This directed tree was denoted in \cite{2012-mams-j-j-s} as $\tcal_{2,\infty}$.)
\begin{figure}[ht]
\begin{center}
\begin{tikzpicture}[scale=0.8, transform shape,edge from parent/.style={draw,to}]
\tikzstyle{every node} = [circle,fill=gray!30]
\node[fill=none] (e-3) at (-6,0) {};
\node (e-2)[font=\footnotesize, inner sep = 1pt] at (-4,0) {$-2$};
\node (e-1)[font=\footnotesize, inner sep = 1pt] at (-2,0) {$-1$};

\node (e10)[font=\footnotesize, inner sep = 3pt] at (0,0) {$0$};

\node (e11)[font=\footnotesize, inner sep = 1pt] at (2,1) {$(1,1)$};
\node (e12)[font=\footnotesize, inner sep = 1pt] at (4,1) {$(1,2)$};
\node (e13)[font=\footnotesize, inner sep = 1pt] at (6,1) {$(1,3)$};
\node[fill = none] (e1n) at(8,1) {};

\node (f11)[font=\footnotesize, inner sep = 1pt] at (2,-1) {$(2,1)$};
\node (f12)[font=\footnotesize, inner sep = 1pt] at (4,-1) {$(2,2)$};
\node (f13)[font=\footnotesize, inner sep = 1pt] at (6,-1) {$(2,3)$};
\node[fill = none] (f1n) at(8,-1) {};

\draw[dashed,->] (e-3) --(e-2) node[pos=0.5,above = 0pt,fill=none] {$1$};
\draw[->] (e-2) --(e-1) node[pos=0.5,above = 0pt,fill=none] {$1$};
\draw[->] (e-1) --(e10) node[pos=0.5,above = 0pt,fill=none] {$1$};
\draw[->=stealth] (e10) --(e11) node[pos=0.5,above = 0pt,fill=none] {$1$};
\draw[->] (e11) --(e12) node[pos=0.5,above = 0pt,fill=none] {$1$};
\draw[->] (e12) --(e13) node[pos=0.5,above = 0pt,fill=none] {$1$};
\draw[dashed, ->] (e13)--(e1n);
\draw[->] (e10) --(f11) node[pos=0.5,below = 0pt,fill=none] {$1$};
\draw[->] (f11) --(f12) node[pos=0.5,below = 0pt,fill=none] {$\alpha$};
\draw[->] (f12) --(f13) node[pos=0.5,below = 0pt,fill=none] {$1$};
\draw[dashed, ->] (f13)--(f1n);
\end{tikzpicture}
\end{center}
\caption{\label{drzewot201}}
\end{figure}
Set $\lambda_{(2,2)}=\alpha$ with some $\alpha\in\cbb$ and $\lambda_u=1$ for all $u\in V^\circ\setminus\{(2,2)\}$. Defined in such a way $\lambdab=\{\lambda_v\}_{v\in{V}}$ induces a $\slam\in\ogr{\ell^2(V)}$. Clearly, for all $V\setminus\{0\}$, $\card{\dzi{v}}=1$ and $\sum_{y\in\dzi{u}}|\lambda_{y}|^2=1$ for $u\in\dzi{v}$. Also, $\dzi{0}=\{(1,1), (2,1)\}$, $\sum_{y\in\dzi{(1,1))}}|\lambda_{y}|^2=1$, and $\sum_{y\in\dzi{(2,1)}}|\lambda_{y}|^2=\alpha$. Thus, by Proposition \ref{suka01}, $\slam$ is weakly centered if and only if $\alpha=1$. Observe, that changing all the other weights $\lambda_v$, $v\in V\setminus\{(1,2), (2,2)\}$ does not affect this.
\end{exa}
\begin{rem}
As easily seen, $\slam$ from Example \ref{blackblack} coincide with a weighted composition operator $\cfw$ with $\phi$ and $w$ induced by $\paa$ and $\lambdab$, respectively. Hence, $\cfw$ is not weakly centered  whenever $\alpha\neq 1$. On the other hand, the (non-weighted) composition operator $C_\phi$ is weakly-entered. Indeed, in view of \cite[Proposition 79]{2018-lnim-budzynski-jablonski-jung-stochel}, we have 
\begin{align}\label{public01}
\hsf_\phi(v)=\card{\dzi{v}},\quad v\in V
\end{align}
which means that $\hsf_\phi$ is constant on all the sets of the form $\phi^{-1}(\{x\})$. $x\in V$. This implies that $\esf_\phi(\hsf_\phi)=\hsf_\phi$ and thus $C_\phi$ is weakly centered by Corollary \ref{debile01}. The reverse situation is also possible, as shown below.
\end{rem}
\begin{exa}\label{blackblack+}
Let $\tcal = (V,E)$ be the directed tree with $$V=\{-k\colon k\in\zbb_+\}\cup\{1, (3,2)\}\cup \big\{1,n), (2,n)\colon n\in\nbb \big\}$$ and $E$ as in as in Figure \ref{truffaz-fig}.\allowdisplaybreaks
\begin{figure}[ht]
\begin{center}
\begin{tikzpicture}[scale=0.8, transform shape,edge from parent/.style={draw,to}]
\tikzstyle{every node} = [circle,fill=gray!30]
\node[fill=none] (e-3) at (-6,0) {};
\node (e-2)[font=\footnotesize, inner sep = 1pt] at (-4,0) {$-1$};
\node (e-1)[font=\footnotesize, inner sep = 3pt] at (-2,0) {$0$};
\node (e10)[font=\footnotesize, inner sep = 3pt] at (0,0) {$1$};

\node (e11)[font=\footnotesize, inner sep = 1pt] at (5,1.5) {$(1,2)$};
\node (e12)[font=\footnotesize, inner sep = 1pt] at (7,1.5) {$(1,3)$};
\node[fill = none] (e1n) at(9,1.5) {};

\node (g11)[font=\footnotesize, inner sep = 1pt] at (2,1.5) {$(1,1)$};

\node (f11)[font=\footnotesize, inner sep = 1pt] at (2,-1.5) {$(2,1)$};
\node (f12)[font=\footnotesize, inner sep = 1pt] at (5,-1.5) {$(2,2)$};
\node (f14)[font=\footnotesize, inner sep = 1pt] at (7,-1.5) {$(2,3)$};
\node[fill = none] (f1n) at(9,-1.5) {};

\node (f13)[font=\footnotesize, inner sep = 1pt] at (5,0) {$(3,2)$};

\draw[dashed,->] (e-3) --(e-2) node[pos=0.5,above = 0pt,fill=none] {$1$};
\draw[->] (e-2) --(e-1) node[pos=0.5,above = 0pt,fill=none] {$1$};
\draw[->] (e-1) --(e10) node[pos=0.5,above = 0pt,fill=none] {$1$};
\draw[->=stealth] (g11) --(e11) node[pos=0.5,above = 0pt,fill=none] {$1$};
\draw[->=stealth] (e10) --(g11) node[pos=0.5,above = 0pt,fill=none] {$1$};

\draw[->] (e10) --(f11) node[pos=0.5,above = 0pt,fill=none] {$1$};
\draw[->] (e11) --(e12) node[pos=0.5,above = 0pt,fill=none] {$1$};
\draw[->] (f11) --(f13) node[pos=0.5,above = 0pt,fill=none] {$0$};
\draw[->] (f11) --(f12) node[pos=0.5,below = 0pt,fill=none] {$1$};
\draw[->] (f12) --(f14) node[pos=0.5,above = 0pt,fill=none] {$1$};
\draw[dashed, ->] (f14)--(f1n);
\draw[dashed, ->] (e12)--(e1n);
\end{tikzpicture}
\end{center}
\caption{\label{truffaz-fig}}
\end{figure} 
Let $\lambdab=\{\lambda_v\}_{v\in{V}}$ satisfy  $\lambda_v=1$ for $v\in V\setminus\{(3,2)\}$ and $\lambda_{(3,2)}=0$. It follows from Proposition \ref{suka01} that $\slam$ is weakly centered. Clearly, $\slam\in\ogr{\ell^2(V)}$ coincide with a weighted composition operator $\cfw$ with $\phi$ and $w$ induced by $\paa$ and $\lambdab$, respectively. Hence, $\cfw$ is weakly centered. By \eqref{public01}, $\hsf_\phi(1,1)=1\neq 2=\hsf_\phi(2,1)$. Thus $C_\phi$ is not weakly centered. 
\end{exa}
\begin{rem}\label{boi01}
One can easily deduce from \cite[Proposition 79]{2018-lnim-budzynski-jablonski-jung-stochel} a characterization of weakly centered wco's on discrete measure spaces. Recall that a measure space $(X, \ascr, \mu)$ is {\em discrete} if $\ascr=2^X$, $\card{\mathsf{At}(\mu)}\leqslant \aleph_0$, $\nu(X\setminus \mathsf{At}(\mu))=0$ and $\mu(\{x\})<\infty$ for all $x\in X$, where $\mathsf{At}(\mu)=\{x\in X\colon \mu(\{x\})>0\}$. In this setting we have $\hfw(x)=\frac{\mu_w(\phi^{-1}(\{x\}))}{\mu(\{x\})}$, $x\in \mathsf{At}(\mu)$. Hence, $\cfw$ is weakly centered, if and only if
\begin{enumerate}
\item[$(\dagger)$] for every $z\in X$ and all $x_1, x_2\in \phi^{-1}(\{z\})$ such that $\mu_w(\{x_1\})\mu_w(\{x_2\})\neq 0$, there is $$\frac{\mu_w(\phi^{-1}(\{x_1\}))}{\mu(\{x_1\})}=\frac{\mu_w(\phi^{-1}(\{x_2\}))}{\mu(\{x_2\})}.$$
\end{enumerate}
In particular, if $\mu$ is a counting measure and $w=1$, then $C_\phi$ is weakly centered if and only if 
\begin{align*}
\text{$\card{\phi^{-1}(\{x_1\})}=\card{\phi^{-1}(\{x_2\})}$ for all $x_1, x_2\in \phi^{-1}(\{z\})$ and $z\in X$}
\end{align*}
(in the directed tree setting: $\card{\dzi{x_1}}=\card{\dzi{x_2}}$ whenever $\pa{x_1}=\pa{x_2}$; cf. Example \ref{blackblack+}).
\end{rem}
The proposition below is a follow up to Proposition \ref{black01} and Remark \ref{blackblack++}. Also, it should be compared to \cite[Lemma 3.3]{2021-lma-holleman-mcclatchy-thompson} which states that {\em an operator $T$ with the polar decomposition $T=U|T|$, where $U$ is unitary, is weakly centered if and only if $|T|$ commutes with $|\Delta_1(T)|$}. Note that if $T=\cfw$, then $|T|$ and $|\Delta_1(T)|$ commute as both are multiplication operators.
\begin{pro}\label{orange}
Assume $\slam\in\ogr{\ell^2(V)}$. Let $\slam=U|\slam|$ be the polar decomposition. If $U$ is unitary, then $\slam$ is unitarily equivalent to a classical two-sided weighted shift. In particular, $\slam$ is centered, i.e. the set $\{T^{*n}T^{n},\ T^{m}T^{*m}\colon n,m\in\zbb_+\}$ is commutative.
\end{pro}
\begin{proof}
In view of \cite[Proposition 3.5.1]{2012-mams-j-j-s}, $U$ coincides with a weighted shift $S_{\widetilde \lambdab}$ and all the weights in $\widetilde\lambdab$ are nonzero. Thus, by \cite[Lemma 8.1.5 and Proposition 8.1.6]{2012-mams-j-j-s}, the directed tree $\tcal$ is isomorphic to $\zbb$ and $\slam$ is unitarily equivalent to a classical two-sided weighted shift.
\end{proof}
\section{Spectrally weakly centered wco's}
The question of commutativity of unbounded operators is a delicate matter. We have two possibilities: pointwise commutativity and spectral commutativity. Recall that operators $A$ and $B$ acting in a Hilbert space $\hh$ {\em pointwise commute} on ${\EuScript D}\subseteq \dz{AB}\cap\dz{BA}$ if and only if $ABf=BAf$ for every $f\in{\EuScript D}$. On the other hand, if $A$ is densely defined and $B$ is normal, we say that $A$ {\em spectrally commutes\footnote{Spectral commutativity is also known as strong commutativity.}} with $B$ if and only $E_B(\sigma)A\subseteq AE_B(\sigma)$ for all $\sigma\in\borel{\cbb}$, where $E_B$ is the spectral measure of $B$. If both $A$ and $B$ are normal, then $A$ spectrally commutes with $B$ if and only $E_B(\sigma)E_A(\omega)= E_A(\omega)E_B(\sigma)$ for all $\sigma,\omega\in\borel{\cbb}$, or equivalently $B$ spectrally commutes with $A$ (in this case, the notion becomes symmetric).

Given a closed densely defined operator $T$ in $\hh$, both the operators $T^*T$ and $T^*$ are self-adjoint, thus we may put forward the following definition: $T$ is {\em spectrally weakly centered} if and only if 
\begin{align*}
E_{T^*T}(\sigma)E_{TT^*}(\omega)= E_{TT^*}(\omega)E_{T^*T}(\sigma),\quad \sigma, \omega \in\borel{\cbb},
\end{align*}
or, equivalently by \cite[Propositions 4.23 and 5.15]{schmudgen},
\begin{align*}
E_{|T|}(\sigma)E_{|T^*|}(\omega)= E_{|T^*|}(\omega)E_{|T|}(\sigma),\quad \sigma, \omega \in\borel{\cbb}.
\end{align*}

Below we answer a question of which unbounded wco's are spectrally weakly centered. As it turns out, the characterization is essentially the same as in the bounded case. We do not use the Alutghe transforms here (see Remark \ref{rudy01}) and rely only on spectral commutativity of $|\cfw|$ and $|\cfw^*|$ (this, by the way, can be an alternative way of proving the main characterization for weakly centered wco's in the bounded case from Theorem \ref{wcent01}).
\begin{thm}\label{gamon01}
Let $\cfw$ be densely defined. The following conditions are equivalent:
\begin{itemize}
\item[(i)] $\hfw=\efw(\hfw)$ a.e. $[\mu_w]$,
\item[(ii)]$\cfw$ is spectrally weakly centered.
\end{itemize}
\end{thm}
\begin{proof}
Assume that (i) is satisfied. This implies that 
\begin{align*}
\varOmega_\sigma\cap\{w\neq 0\}\in \phi^{-1}(\ascr)\cap\{w\neq 0\},\quad \sigma\in\borel{\cbb},
\end{align*}
where $\varOmega_\sigma:=\big(\sqrt{\hfw}\big)^{-1}(\sigma)$. In particular, we get
\begin{align}\label{unb02}
w\efw\big(\chi_{\varOmega_\sigma} g_w\big)=w\chi_{\varOmega_\sigma} \efw(g_w), \quad \sigma\in\borel{\cbb}, g\in L^2(\mu),
\end{align}
where $g_w=\chi_{\{w\neq0\}}\frac{g}{w}$ (cf. \cite[2.12]{2018-lnim-budzynski-jablonski-jung-stochel}).
The spectral measure $E=E_{|\cfw|}$ of $|\cfw|$ is given by (see \cite[Ex. 5.3, pg. 93]{schmudgen})
\begin{align*}
E(\sigma)g=\chi_{\varOmega_\sigma} g,\quad g\in L^2(\mu).
\end{align*}
Fix $\sigma\in\borel{\cbb}$. Recall that $\dz{|\cfw^*|}=\big\{g\in L^2(\mu)\colon w\,(\hfw\circ\phi)^{1/2}\efw(g_w)\in L^2(\mu)\big\}$ (see \cite[Theorem 18]{2018-lnim-budzynski-jablonski-jung-stochel}). Let $f\in\dz{|\cfw^*|}$. Then, by \eqref{unb02} and \cite[Theorem 18]{2018-lnim-budzynski-jablonski-jung-stochel}, we have
\begin{align*}
w\,(\hfw\circ\phi)^{\frac12}\efw\big((\chi_{\varOmega_\sigma}f)_w\big)&=
w\,(\hfw\circ\phi)^{\frac12}\efw\big(\chi_{\varOmega_\sigma}f_w\big)\\
&= \chi_{\varOmega_\sigma} w\,(\hfw\circ\phi)^{\frac12}\efw(f_w) =E(\sigma)|\cfw^*|f\in L^2(\mu).
\end{align*}
Thus $E(\sigma)f\in \dz{|\cfw^*|}$ and  $E(\sigma)|\cfw^*|f=|\cfw^*|E(\sigma)f$. Since $f$ can be arbitrarily chosen, we get $E(\sigma)|\cfw^*|\subseteq |\cfw^*|E(\sigma)$. This and \cite[Proposition 5.15]{schmudgen} yield $E_{|\cfw|}(\sigma)E_{|\cfw^*|}(\omega)=E_{|\cfw^*|}(\omega)E_{|\cfw|}(\sigma)$ for all $\omega\in\borel{\cbb}$. Consequently, we see that $E_{|\cfw|}(\sigma)$ and $E_{|\cfw^*|}(\omega)$ commute for all $\sigma, \omega\in\borel{\cbb}$. Hence, (ii) holds.

Now, we assume (ii). This implies that $E_{|\cfw|}(\sigma)|\cfw^*|\subseteq |\cfw^*|E_{|\cfw|}(\sigma)$ for every $\sigma\in\borel{\cbb}$. Hence we deduce
\begin{align}\label{unb03}
    \chi_{\varOmega_\sigma} w\, (\hfw\circ\phi)^{1/2} \efw(f_w)=w\, (\hfw\circ\phi)^{1/2} \efw\big(\chi_{\varOmega_\sigma}f_w\big),\quad f\in \dz{|\cfw^*|}.
\end{align}
Let $\varDelta\in\ascr$ satisfy
\begin{align}\label{unb04}
\chi_\varDelta+\chi_\varDelta (\hfw\circ\phi) \in L^1(\mu).
\end{align}
We then have
\allowdisplaybreaks
\begin{align*}
\int_X&\big|w\, (\hfw\circ\phi)^{\frac12}\efw((\chi_\varDelta)_w)\big|^2\D\mu=\int_X|w|^2 (\hfw\circ\phi) \big|\efw((\chi_\varDelta)_w)\big|^2\D\mu\\
&\leqslant \int_X|w|^2 (\hfw\circ\phi) \efw\big(\big|(\chi_\varDelta)_w\big|^2\big)\D\mu
=\int_X (\hfw\circ\phi) \efw\big(\big|(\chi_\varDelta)_w\big|^2\big)\D\mu_w\\
&=\int_X (\hfw\circ\phi) \big|(\chi_\varDelta)_w\big|^2\D\mu_w
=\int_X (\hfw\circ\phi) \big|(\chi_\varDelta)_w\big|^2\D\mu_w\\
&=\int_X (\hfw\circ\phi) \chi_{\varDelta\cap\{w\neq 0\}} \D\mu<\infty.
\end{align*}
This implies that $\chi_\varDelta\in\dz{|\cfw^*|}$ for any $\varDelta$ such that \eqref{unb04} holds. Since $(X,\ascr, \mu)$ is $\sigma$-finite and $\hsf\circ \phi<\infty$ a.e. $[\mu]$ (use \cite[Proposition 10]{2018-lnim-budzynski-jablonski-jung-stochel}), we deduce that there is a sequence $\{\varDelta_n\}_{n=1}^\infty$ of sets satisfying \eqref{unb04} such that $\varDelta_k\subseteq \varDelta_{k+1}$ for every $k\in\nbb$ and $\bigcup_{n=1}^\infty\varDelta_n=X$. Using \eqref{unb03} we get
\begin{align*}
    \chi_{\varOmega_\sigma} \efw(\chi_{\varDelta_n})= \efw(\chi_{\varOmega_\sigma}\chi_{\varDelta_n})\text{ a.e. $[\mu_w]$},\quad n\in\nbb, \sigma\in\ascr.
\end{align*}
Since $\chi_{\varDelta_n}\nearrow \chi_X$ as $n\to+\infty$ and $\efw{\chi_X}=\chi_X$, we get
\begin{align*}
    \chi_{\varOmega_\sigma}= \efw(\chi_{\varOmega_\sigma})\text{ a.e. $[\mu_w]$}\quad  \sigma\in\borel{\cbb}.
\end{align*}
This implies (i).
\end{proof}
Spectral commutativity of selfadjoint operators implies pointwise commutativity. The question of whether the reverse holds in general has been answered negatively by Nelson in \cite{1959-am-nelson} (the unilateral shift serves as a basis for a construction of a counterexample). However, if some additional assumptions are imposed, the answer is positive -- see the Nelson criterion for spectral commutativity \cite[Corollary 9.1]{1959-am-nelson}. Below we provide a criterion for $\cfw$ to be spectrally weakly centered based on pointwise commutativity of $|\cfw|$ and $|\cfw^*|$. Its proof uses an operator $\pfw$ in $L^2(\mu)$ given by 
\begin{align*}
\pfw f=w\efw(f_w),\quad f\in L^2(\mu).
\end{align*}
As shown in \cite[Lemma 4.2]{2020-mn-benhida-budzynski-trepkowski}, $\pfw$ is an orthogonal projection.
\begin{pro}\label{gamon01+}
Let $\cfw$ be densely defined. The following conditions are equivalent:
\begin{itemize}
\item[(i)]$\cfw$ is spectrally weakly centered,
\item[(ii)] there exists ${\EuScript D}\subseteq \dz{|\cfw||\cfw^*|}\cap\dz{|\cfw^*||\cfw|}\cap\dz{M_{\hfw}}$, a dense subspace of $L^2(\mu)$, such that $|\cfw||\cfw^*|f=|\cfw^*||\cfw|f$ for every $f\in {\EuScript D}$.
\end{itemize}
\end{pro}
\begin{proof}
The implication (i)$\Rightarrow$(iii) holds by \cite[Corollary 5.28]{schmudgen}.

Assume (ii). Consider two operators $\psf=\pfw$ and $M=|\cfw|=M_{\sqrt{\hfw}}$. It is clear that ${\EuScript D}\subseteq \dz{\psf M}\cap\dz{\psf^2}\cap\dz{M^2}$. Moreover, ${\EuScript D}\subseteq \dz{M\psf}$. Indeed, by pointwise commutativity of $|\cfw|$ and $|\cfw^*|$, and \cite[Theorem 18]{2018-lnim-budzynski-jablonski-jung-stochel} we have
\begin{align}\label{koordynator01}
w\sqrt{\hfw}\efw(f_w)=w\efw\Big(\sqrt{\hfw} f_w\Big),\quad f\in {\EuScript D}.
\end{align}
Thus we get\allowdisplaybreaks
\begin{align*}
\int_X\hfw \Big| w\efw\big( f_w\big)\Big|^2\D\mu
&=\int_X |w|^2\Big| \efw\big(\sqrt{\hfw} f_w\big)\Big|^2\D\mu
=\int_X \Big| \efw\big(\sqrt{\hfw}f_w\big)\Big|^2\D\mu_w\\
&\leqslant \int_X \efw\Big(\hfw |f_w|^2\Big)\D\mu_w
=\int_X \chi_{\{w\neq0\}}\hfw |f|^2\D\mu<+\infty.
\end{align*}
(Recall that ${\EuScript D}\subseteq \dz{|\cfw|}=L^2((1+\hfw)\D\mu)$.) This proves that $\psf f\in\dz{M}$ for every $f\in{\EuScript D}$, consequently ${\EuScript D}\subseteq \dz{M\psf}$. Hence ${\EuScript D}\subseteq \dz{\psf M}\cap\dz{M\psf}\cap\dz{\psf^2}\cap\dz{M^2}$. In view of \eqref{koordynator01}, we have
\begin{align*}
\psf Mf=M\psf f,\quad f\in {\EuScript D}.
\end{align*}
We now show that $A=(\psf ^2+M^2)|_{{\EuScript D}}$ is essentially selfadjoint, i.e. $A^*=\bar A$, where $\bar A$ is the closure of $A$. Clearly, $A=(\psf^2 +M^2)|_{{\EuScript D}}=\psf|_{{\EuScript D}}+M^2|_{{\EuScript D}}=\psf +N$, with $Nf=\hfw f$ for $f\in{\EuScript D}$. Since $\psf
\in\ogr{L^2(\mu)}$, essential selfadjointness of $A$ is equivalent to that of $N$. Since $\bar N\subseteq \overline{M_{\hfw}}=M_{\hfw}$ we get $\bar N\subseteq M_{\hfw}\subseteq {\bar N}^*$ which means that $\bar N$ is symmetric. Moreover, ${\mathcal N}(N^*-{\mathrm i}I)=\{0\}$. Indeed, for any $f\in {\mathcal N}(N^*-{\mathrm i}I)$ we have
\begin{align*}
\is{{\mathrm i} f}{g}=\is{N^*f}{g}=\is{f}{M_{\hfw}g}=\is{f}{\hfw g}=\is{\hfw f}{g},\quad g\in {\EuScript D}.
\end{align*}
Hence
\begin{align*}
\int_{X} ({\mathrm i}-\hfw) f\bar g\D\mu=0,\quad g\in {\EuScript D}.
\end{align*}
Since ${\EuScript D}$ is dense in $L^2(\mu)$, we get that $({\mathrm i}-\hfw) f=0$ a.e. $[\mu]$. This implies that $f=0$ a.e. $[\mu]$ and so ${\mathcal N}(N^*-{\mathrm i}I)=\{0\}$ is proved. Now, by \cite[Proposition 3.8]{schmudgen}, $N$ is essentially selfadjoint and thus $A$ is essentially selfadjoint. Hence, in view of the Nelson criterion, $\psf$ spectrally commutes with $M$. Therefore
\begin{align}\label{tomo01}
\psf E_M(\sigma)=E_M(\sigma)\psf ,\quad \sigma\in\borel{\cbb},
\end{align}
where $E_M$ is the spectral measure for $M$. Since $E_M(\sigma) g=\chi_{\varOmega_\sigma} g$ for $g\in L^2(\mu)$, with $\varOmega_\sigma=\big(\sqrt{\hfw}\big)^{-1}(\sigma)$, we deduce from \eqref{tomo01} that
\begin{align*}
\efw\Big(\chi_{\varOmega_\sigma} g_w\Big)=\chi_{\varOmega_\sigma}\efw(g_w) \text{ a.e. }[\mu_w], \quad g\in L^2(\mu).
\end{align*}
From this we get $\chi_{\varOmega_\sigma} =\efw \big(\chi_{\varOmega_\sigma}\big)$ a.e. $[\mu_w]$ for  every $\sigma\in\borel{\cbb}$. This implies (i) and completes the proof.
\end{proof}
The above yields an interesting problem.
\begin{opq}
Could the condition (ii) of Proposition \ref{gamon01+} be replaced by
\begin{enumerate}
\item[(iii)] there exists ${\EuScript D}\subseteq \dz{|\cfw||\cfw^*|}\cap\dz{|\cfw^*||\cfw|}$, a dense subspace of $L^2(\mu)$, such that $|\cfw||\cfw^*|f=|\cfw^*||\cfw|f$ for every $f\in {\EuScript D}$?
\end{enumerate}
In another words, the question is whether ``pointwise'' weakly centered wco's are spectrally weakly centered.
\end{opq}
\begin{rem}
In view of Theorem \ref{gamon01}, all the results of the previous section implied by Theorem \ref{wcent01} are valid in the context of unbounded wco's as long as {\em weakly centered} is replaced by {\em spectrally weakly centered} and the operators are assumed to be densely defined. The above applies to examples as well. In particular, Example \ref{sad04} provides a large class of spectrally weakly centered operators that are not bounded. For this it suffices to drop the assumptions that {\em $\rho$ is a polynomial or $\|\phi^{-1}\|\leqslant 1$}.
\end{rem}
\begin{rem}\label{rudy01}
Regarding possibility of using the Alutghe transform in the context of unbounded wco's it is worth pointing out that it may be essentially more complicated than in the bounded case. For example, the transform of a hyponormal composition operator may have trivial domain as was shown in \cite{2015-jmaa-trepkowski}. Such a pathological example is not possible if the transformed operator is spectrally weakly centered wco - the Alutghe transform of a densely defined wco is densely defined (this follows from Theorem \ref{gamon01} and \cite[Corollary 3.9]{2020-mn-benhida-budzynski-trepkowski}). However, it may happen that $\Delta_{1/2}(\cfw)\neq C_{\phi, w_{1/2}}$. 

Indeed, consider $(X,\ascr, \mu)=(\zbb, 2^\zbb, \nu)$ with $\nu$ being the counting measure. Let $\phi\colon X\to X$ be given by $\phi(n)=n-1$. Finally, let $w\colon X\to \cbb$ be defined by $w(3n)=1$, $w(3n+1)=n^2$, $w(3n+2)=\frac{1}{n^2}$ for all $n\in\nbb$, and $w(k)=1$ for all $k\in \zbb\setminus\nbb$. Then $\cfw$ is densely defined. Since $\phi$ is a bijection, we have $\efw(\hfw)=\hfw$, which implies that $\cfw$ is spectrally weakly centered. On the other hand, it is easily seen that there is no positive constant $c$ that would satisfy
\begin{align*}
\sqrt{\hsf(x)}\leqslant c\Big(1+\efw\big(\sqrt{\hfw}\big)\circ\phi^{-1}(x)\sqrt{\hfw(x)}\Big),\quad x\in X,
\end{align*}
(for $x=3n$, $n\in\nbb$, the left hand side is equal to $n$ while the right hand side is equal to $2c$), which implies that $\Delta_{1/2}(\cfw)\neq C_{\phi,w_{1/2}}$ (see \cite[Theorem 3.2]{2020-mn-benhida-budzynski-trepkowski}). Moreover, $\Delta_{1/2}(\cfw)$ is not even a closed operator.
\end{rem}
\section{Invariant subspaces}
A closed subspace $\mathcal{M}$ of a Hilbert space $\hh$ is called an {\em invariant subspac}e for an operator $T\in \ogr{\hh}$ if $T\mathcal{M}\subseteq \mathcal{M}$. The ones that are of interest are different from $\{0\}$ and $\hh$; we call them {\em nontrivial}. Recall a result on invariant subspaces due to Campbell (see \cite[Theorem 2]{1972-pams-campbell}).
\begin{thm}\label{stach01}
Let $T\in\ogr{\hh}$ be hyponormal and weakly centered. Then there exists a nontrivial invariant subspace of $T$.
\end{thm}
Using the above and Theorem \ref{wcent01} we deduce  the following criterion for existence of invariant subspaces for wco's.
\begin{pro}\label{sad05}
Let $\cfw\in \ogr{L^2(\mu)}$ satisfy $\efw(\hfw)=\hfw$ a.e. $[\mu_w]$ and $\hfw\circ \phi\leqslant\hfw$ a.e. $[\mu_w]$. Then $\cfw$ has a nontrivial invariant subspace.
\end{pro}
\begin{proof}
According to Theorem \ref{stach01}, $\cfw$ has a nontrivial invariant subspace whenever $\cfw$ hyponormal and weakly centered. In view of \ref{wcent01}, $\efw(\hfw)=\hfw$ a.e. $[\mu_w]$ implies that $\cfw$ is weakly centered. In turn, $0<\hfw\circ \phi\leqslant\hfw$ a.e. $[\mu_w]$ imply that $\cfw$ is hyponormal by \cite[Theorem 53]{2018-lnim-budzynski-jablonski-jung-stochel} ($0<\hfw\circ \phi$ a.e. $[\mu_w]$ holds due to \cite[Lemma 6]{2018-lnim-budzynski-jablonski-jung-stochel}). 
\end{proof}
\begin{exa}
Let $C_\phi\in\ogr{L^2(\mu)}$ be as in Example \ref{sad04}. If $\phi^{-1}$ is a contraction on $\rbb^n$, then $\phi^{-1}(B)\subseteq B$, where $B$ denotes a ball $\{x\in\rbb^n\colon \|x\|\leqslant r\}$ with $r>0$. Consequently, by \cite[Proposition 2.4]{2017-bims-azi-jab-jaf}, $C_\phi$ has a nontrivial invariant subspace. When $\phi^{-1}$ is not a contraction (if this is the case, $\rho$ has to be a polynomial for $C_\phi$ to be bounded), one is left with Proposition \ref{sad05} when looking for examples of $C_\phi$'s with invariant spaces. 

By the change of the variables theorem we get
\begin{align*} 
\hsf_\phi(x)= \frac{1}{|\det \phi|} \frac{\rho\big(\|\phi^{-1}(x)\|^2\big)}{\rho\big(\|x\|^2\big)} \quad \text{ for $\mu$-a.e. } x \in \rbb^n.
\end{align*}
Thus, by Proposition \ref{sad05}, $C_\phi$ has a nontrivial invariant space whenever 
\begin{align}\label{sad06}
\text{$\rho^2\big(\|x\|^2\big)\leqslant\rho\big(\|\phi^{-1}(x)\|^2\big)\rho\big(\|\phi(x)\|^2\big)$ for $\M_n$-almost every $x\in\rbb^n$.}
\end{align}

Consider $\phi(x)=\alpha x$, $x\in\rbb$, with $0<\alpha<1$, and $\rho(z)=Az+B$, $z\in\cbb$, with $A, B>0$. Then $\phi^{-1}$ is not a contraction and so \cite[Proposition 2.4]{2017-bims-azi-jab-jaf} is not applicable. On the other hand, one can easily check that \eqref{sad06} is equivalent to
\begin{align*}
\big(\alpha^2 + \frac{1}{\alpha^2}-2\big)t\geqslant 0 \text{ for } t\geqslant 0.
\end{align*}
The latter is satisfied for all $\alpha\in\rbb$. This means that $C_\phi$ has to have a nontrivial invariant subspace (note that this is true independent of choice of $A$ and $B$).
\end{exa}

We now turn our attention to unbounded operators. If $T$ is an operator in a Hilbert space $\hh$ and $\mathcal{M}$ is a closed subspace of $\hh$ such that $T(\mathcal{M}\cap \dz{T})\subseteq \mathcal{M}$, then $\mathcal{M}$ is called an {\em invariant space} for $T$. In view of Proposition \ref{sad05} and results on spectrally weakly centered wco's, it is natural to ask wether anything can be said about invariant subspaces for unbounded wco's.

A close examination of the proof of Theorem \ref{stach01} reveals its two important components: one is the equality $F_A(x)T=TF_B(x)$ and the other is the inequality $F_B(x)\leqslant F_A(x)$, where $F_A$ and $F_B$ denote spectral distributions of $A=TT^*$ and $B=T^*T$, respectively. Recall that $F_Y(x)=E_Y\big((-\infty, x]\big)$ for $x\in\rbb$ and $E_Y$ is the spectral measure for $Y\in\{A,B\}$. We will investigate these two in the context of unbounded wco's.

Suppose that $\cfw$ is densely defined. Let us denote $A=\cfw\cfw^*$ and $B=\cfw^*\cfw$. Both $A$ and $B$ are positive selfadjoint. Due to \cite[Theorem 18]{2018-lnim-budzynski-jablonski-jung-stochel} and \cite[Ex. 5.3, pg. 93]{schmudgen}, the spectral measure of $B$ is given by
\begin{align}\label{staszek01}
E_B(\sigma)=M_{\chi_{\varDelta_\sigma}},\quad \sigma\in\borel{\rbb},
\end{align}
with $\varDelta_\sigma=\hfw^{-1}(\sigma)$. Now we calculate the spectral measure $E_A$. For this we will use the fact that $A= M \psf$, where $M=M_{\hfw\circ\phi}$ (see \cite[Theorem 4.3]{2020-mn-benhida-budzynski-trepkowski}) and $\psf=\pfw$. The spectral measure of $M$ is $E_M(\sigma) = M_{\chi_{\varGamma_\sigma}}$ with $\varGamma_\sigma=(\hfw\circ\phi)^{-1}(\sigma)$ for all $\sigma\in\borel{\rbb}$. Since $\psf $ is an orthogonal projection such that $\psf M\subseteq M\psf $, the spectral measure of $A$ is given by (see \cite[proof of Lemma 4.1]{2020-mn-benhida-budzynski-trepkowski})
\begin{align}\label{staszek02}
E_A(\sigma)=E_M(\sigma)\psf +\chi_{\sigma}(0)(I-\psf )=M_{\chi_{\varGamma_\sigma}}\psf +\chi_{\sigma}(0)(I-\psf ),\quad \sigma\in\borel{\rbb}.
\end{align}

We are ready to prove the following.
\begin{lem}\label{staszek03l}
Assume that $\cfw$ densely defined and $\jd{\cfw^*}=\{0\}$. Then the following holds
\begin{align}\label{staszek03}
E_A(\sigma) \cfw \subseteq \cfw E_B(\sigma),\quad \sigma\in \rbb.
\end{align}
\end{lem}
\begin{proof}
For this we fix $\sigma\in\rbb$. It is obvious that the domain of the left hand side operator of the inclusion above is equal to $\dz{\cfw}=\{f\in L^2(\mu)\colon \int_X |f|^2\hfw\D\mu<\infty\}$. On the other hand, $f\in \dz{\cfw E_B(\sigma)}$ if and only if $E_B(\sigma)f\in \dz{\cfw}$ or, equivalently, $\int_X|E_B(\sigma) f|^2\hfw\D\mu<\infty$. In view of \eqref{staszek01}, we have
\begin{align*}
\int_X|E_B(\sigma) f|^2\hfw\D\mu\leqslant \int_X|f|^2\hfw\D\mu <+\infty,\quad f\in\dz{\cfw}.
\end{align*}
Therefore, we see that
\begin{align}\label{staszek06}
\dz{E_A(\sigma) \cfw} \subseteq \dz{\cfw E_B(\sigma)}.
\end{align}
By \eqref{staszek02}, for every $f\in\dz{\cfw}$ we have
\begin{align}\label{staszek05}
E_A(\sigma) \cfw f&=\chi_{\varGamma_\sigma}\psf (w f\circ\phi)+\chi_{\sigma}(0)(I-\psf )(wf\circ\phi)\notag\\
&=\chi_{\varGamma_\sigma}w (f\circ\phi) \efw(\chi_{\{w\neq 0\}})+\chi_{\sigma}(0) w(f\circ\phi) - \chi_{\sigma}(0) w (f\circ\phi)\efw(\chi_{\{w\neq 0\}}).
\end{align}
On the other hand, by \eqref{staszek01} and \eqref{staszek06}, for every $f\in\dz{\cfw}$ we have
\begin{align}\label{staszek07}
\cfw E_B(\sigma)f&=w (f \circ \phi) (\chi_{\varDelta_\sigma}\circ\phi)= \chi_{\varGamma_\sigma} w (f \circ \phi).
\end{align}
Comparing \eqref{staszek05} and \eqref{staszek07} we see that
\begin{align}\label{staszek08}
E_A(\sigma) \cfw f= \cfw E_B(\sigma)f,\quad f\in\dz{\cfw},
\end{align}
whenever $w\neq0$ a.e. $[\mu]$. Since $\chi_{\{w=0\}}L^2(\mu)\subseteq \jd{\cfw^*}$ (by \cite[Proposition 17]{2018-lnim-budzynski-jablonski-jung-stochel}) and $\sigma\in \borel{\rbb}$ can be arbitrary, we deduce from \eqref{staszek06} and \eqref{staszek08} that \eqref{staszek03} holds whenever $\jd{\cfw^*}=\{0\}$.
\end{proof}
\begin{rem}
Regarding \eqref{staszek06}, it may be worth noticing that a similar looking inclusion
\begin{align}\label{staszek03'}
\dz{E_B(\sigma) \cfw }\subseteq \dz{\cfw E_A(\sigma)},\quad \sigma\in \borel{\rbb}
\end{align}
holds whenever $\cfw$ is spectrally weakly centered. This can be proved in the following way. We fix $\sigma\in \borel{\rbb}$. Clearly, $\dz{E_B(\sigma) \cfw}=\dz{\cfw}$. By \eqref{staszek02}, we see that $\{f\in \dz{\cfw}\colon \psf f\in \dz{\cfw}\}\subseteq \dz{\cfw E_A(\sigma)}$. Hence, it suffices to show that
\begin{align}\label{niemiecki01}
\dz{\cfw}\subseteq \{f\in \dz{\cfw}\colon \psf  f\in \dz{\cfw}\}.
\end{align}
The above holds whenever $\int_X |\psf f|^2\hfw\D\mu<\infty$ for every $f\in\dz{\cfw}$. Since $\hfw=\efw(\hfw)$ a.e. $[\mu_w]$, as $\cfw$ is assumed to be spectrally weakly centered, we have
\begin{align*}
\int_X |\psf f|^2\hfw\D\mu&\leqslant\int_X \efw\big(|f_w|^2\big)\hfw\D\mu_w\\
&=\int_X|f_w|^2\hfw\D\mu_w\leqslant\int_X|f|^2\hfw\D\mu<+\infty, \quad f\in\dz{\cfw}
\end{align*}
Hence, \eqref{niemiecki01} is satisfied. As a consequence, we get the following \eqref{staszek03'}.
\end{rem}
Our next step is to prove the following
\begin{lem}\label{staszek09l}
Assume $\cfw$ is spectrally weakly centered and hyponormal. Then the following condition holds \begin{align}\label{staszek09}
F_B(t)\leqslant F_A(t),\quad t\in \rbb.
\end{align}
\end{lem}
\begin{proof}
Condition \eqref{staszek09}, by definition, means that $A\preccurlyeq B$, where ''$\preccurlyeq$'' is Olson's spectral order (see \cite{1971-pams-olson} for a bounded case and \cite{2012-jmaa-planeta-stochel} for an unbounded one). According to \cite[Corollary 9.6]{2012-jmaa-planeta-stochel}, $A\preccurlyeq B$ is satisfied if and only if $A\leqslant B$ and $F_A(t)F_B(t)=F_B(t)F_A(t)$ for every $t\in\rbb$. The latter condition is holds since $\cfw$ is spectrally weakly centered. The former follows from hyponormality of $\cfw$. Indeed, we first note that 
\begin{align}\label{prosiak01}
\dz{B^{\frac12}}=\dz{|\cfw|}=\dz{\cfw}\subseteq\dz{\cfw^*}= \dz{|\cfw^*|}=\dz{A^{\frac12}}.
\end{align}
Moreover, we have
\begin{align}\label{prosiak02}
\|A^{\frac12}f\|=\|U|\cfw^*|f\|=\|\cfw^* f\|\leqslant\|\cfw f\|=\|V|\cfw|f\|=\|B^{\frac12}f\|, \quad f\in \dz{A^{\frac12}},
\end{align}
where $U$ and $V$ are the phases of $\cfw^*$ and $\cfw$, respectively. Both \eqref{prosiak01} and \eqref{prosiak02} together imply that $A\leqslant B$. Consequently, \eqref{staszek09} is satisfied.
\end{proof}
We now are ready to prove a criterion for existence of invariant subspaces for unbounded wco's.
\begin{thm}\label{smutek01}
Let $\cfw$ be densely defined, $\efw(\hfw)=\hfw$ a.e. $[\mu_w]$, and $\hfw\circ\phi\leqslant \hfw$ a.e. $[\mu_w]$. Then there is a nontrivial invariant subspace $\mathcal{M}$ for $\cfw$. If kernels of $\cfw$ and $\cfw^*$ are trivial, and $\cfw$ is not a scalar multiple of an isometry, then the subspace $\mathcal{M}$ is of the form $\chi_{\{\hfw>t_0\}}L^2(\mu)$ for some $t\in\rbb$.
\end{thm}
\begin{proof}
By Theorem \ref{gamon01} and \cite[Theorem 53]{2018-lnim-budzynski-jablonski-jung-stochel}, $\cfw$ is  spectrally weakly centered and hyponormal. We may assume that $\cfw\neq 0$ and $\cfw$ is not a scalar multiple of an isometry as in this case the claim is satisfied. We also may assume that the kernel $\jd{\cfw}=\{0\}$ since otherwise $\jd{\cfw}$ is a nontrivial invariant subspace of $\cfw$.

Now assume that $\jd{\cfw^*}\neq 0$. Since $L^2(\mu)=\jd{\cfw^*}\oplus \overline{\mathcal{R}(\cfw)}$ and $\cfw f\in \overline{\mathcal{R}(\cfw)}$ for every $f\in \dz{\cfw}$ it suffices to show that $\dz{\cfw^2}\neq \{0\}$ to prove that $\overline{\mathcal{R}(\cfw)}$ is a nontrivial invariant subspace we look for. Suppose that $\dz{\cfw^2}=\{0\}$. In view of \cite[Lemma 44]{2018-lnim-budzynski-jablonski-jung-stochel}, we have $\overline{\cfw^2}=C_{\phi^2,\hat w_2}$, where $\phi^2=\phi\circ \phi$ and $\hat w_2=w\cdot (w\circ\phi)$. Thus $\dz{C_{\phi^2,\hat w_2}}=\{0\}$. This implies that $\hsf_{\phi^2,\hat w_2}=\infty$ a.e. $[\mu]$ (see \cite[Proposition 8]{2018-lnim-budzynski-jablonski-jung-stochel}). Consequently, $\hsf_{\phi^2,\hat w_2}=\infty$ a.e. $[\hfw\D\mu]$ and so, by \cite[Lemma 5]{2018-lnim-budzynski-jablonski-jung-stochel}, we see that $\hsf_{\phi^2,\hat w_2}\circ \phi=\infty$ a.e. $[\mu_w]$. This is in clear contradiction with
\begin{align}\label{gamon02}
\hsf_{\phi^2,\hat w_2}\circ \phi= \efw(\hfw)\cdot (\hfw\circ\phi)\leqslant \hfw^2\quad \text{a.e $[\mu_w]$},
\end{align}
which follows from Theorem \ref{gamon01}, \cite[Theorem 53]{2018-lnim-budzynski-jablonski-jung-stochel}, and \cite[Lemma 26]{2018-lnim-budzynski-jablonski-jung-stochel}, and the fact that $\hfw<\infty$ a.e. $[\mu]$, which holds since $\cfw$ is densely defined. Therefore $\dz{\cfw^2}\neq \{0\}$ and thus ${\mathcal{R}(\cfw)}$ is a nontrivial invariant subspace of $\cfw$. Henceforth, we may assume that $\jd{\cfw^*}= \{0\}$. 

For $t\in\rbb_+$ we set $\mathcal{M}_t:=\jd{E_B((-\infty,t])}=\jd{F_B(t)}$. By Lemmas \ref{staszek03l} and \ref{staszek09l}, we get 
\begin{align*}
\cfw \big(\dz{\cfw}\cap \mathcal{M}_t\big)\subseteq \jd{F_A(t)} \subseteq \mathcal{M}_t,\quad t\in\rbb_+.
\end{align*}
We now see that finishing the proof amounts to showing that for some $t_0\in\rbb_+$, $\dz{\cfw}\cap \mathcal{M}_{t_0}$ is nontrivial and that $\mathcal{M}_{t_0}$ is of the form $\chi_{\{\hfw>t_0\}}L^2(\mu)$. The latter follows immediately from \eqref{staszek01}. The former can be proved as follows. Let $t_0\in\rbb_+$ be so as $\mu\big(\{\hfw>t_0\}\big)>0$ and $\mu\big(X\setminus\{\hfw>t_0\}\big)>0$. Such a $t_0$ exists as otherwise $\hfw$ is constant a.e. $[\mu]$ and this means that $\cfw$ is a scalar multiple of an isometry. Clearly $\{0\}\neq \mathcal{M}_{t_0}\neq L^2(\mu)$. Moreover, $\dz{\cfw}\cap \mathcal{M}_{t_0}=\chi_{\{\hfw>t_0\}}\dz{\cfw}\neq\{0\}$. Indeed, if $\chi_{\{\hfw>t_0\}}\dz{\cfw}=\{0\}$, then for every $f\in \dz{\cfw}$ we have $f=\chi_{\{\hfw\leqslant t_0\}}f$. However, since $0<\mu\big(\{\hfw>t_0\}\big)$ and $\mu$ is $\sigma$-finite, there exists $\omega\subseteq \{\hfw>t_0\}$ such that $0<\mu(\omega)<+\infty$. The Radon-Nikodym derivative $\hfw$ is finite a.e. $[\mu]$ so $\{\hfw\leqslant n\}\nearrow X$ as $n\to+\infty$. This means the for some $n_0\in\nbb$, the set $\varOmega=\omega\cap \{\hfw\leqslant n_0\}$ satisfies $0<\mu(\varOmega)<+\infty$. If that is the case, then $g=\chi_\varOmega\in \dz{\cfw}$, $g=\chi_{\{\hfw>t_0\}}g$ and $g\neq 0$. This contradicts $\chi_{\{\hfw>t_0\}}\dz{\cfw}=\{0\}$ and so the proof is complete.
\end{proof}
\section{Acknowledgments}
I wish to express my deep gratitude to Professor Marek Ptak and his wife for their hospitality and support during and beyond preparation of this paper. 


\bibliographystyle{amsalpha}

\begin{thebibliography}{99}
\bibitem{1990-ieot-aluthge} A. Aluthge, On $p$-hyponormal operators for $0 < p < 1$, Integr. Equ. Oper. Theory, {\bf 13} (1990), 307-315.
\bibitem{ash} R. B. Ash, Probability and measure theory, Harcourt/Academic Press, 2000.
\bibitem{2017-bims-azi-jab-jaf} M.R. Azimi, M.R. Jabbarzadeh, M. Jafari Bakhshkandi, On reducibility of weighted composition operators, Bulletin Iranian Math. Soc. {\bf 43} (2017), 875-883.
\bibitem{2022-mjm-azimi-jabbarzadeh} M.R. Azimi, M.R. Jabbarzadeh, Hypercyclicity of weighted composition operators on $L^p$-spaces, Mediterr. J. Math. {\bf 164} (2022) 19-164.
\bibitem{2020-mn-benhida-budzynski-trepkowski} C. Benhida, P. Budzy\'nski, J. Trepkowski, Authge transforms of unbounded weighted composition operators in $L^2$-spaces, Math. Nachr. {\bf 293} (2020), 1888-1910.
\bibitem{2024-rim-budzynski} P. Budzyński, Reciprocals of weighted composition operators in $L^2$-spaces, Results in Math. {\bf 79}, 118 (2024),
\bibitem{2015-jfa-budzynski-jablonski-jung-stochel} P. Budzy\'{n}ski, Z. J. Jab{\l}o\'nski, I. B. Jung, J. Stochel, Unbounded subnormal composition composition operators in $L^2$-spaces, J. Funct. Anal. {\bf 269} (2015), 2110-2164.
\bibitem{2017-aim-budzynski-jablonski-jung-stochel} P. Budzy\'{n}ski, Z. J. Jab{\l}o\'nski, I. B. Jung, J. Stochel, Subnormality of unbounded composition operators over one-circuit directed graphs: Exotic examples, Advances Math. {\bf 310} (2017), 484-556.
\bibitem{2018-lnim-budzynski-jablonski-jung-stochel} P. Budzy\'{n}ski, Z. J. Jab{\l}o\'nski, I. B. Jung, J. Stochel, Unbounded weighted composition operators in $L^2$-spaces, Lectures Notes in Mathematics {\bf 2209} (2018), Springer.
\bibitem{1972-pams-campbell} S. L. Campbell, Linear operators for which $T^*T$ and $TT^*$ commute, Proc. Amer. Math. Soc. {\bf 34} (1972), 177-180.
\bibitem{c-pjm-1974} S. L. Campbell, Linear operators for which $T^*T$ and $TT^*$ commute. II, Pacific J. Math. {\bf 53} (1974), 355-361.
\bibitem{2005-ieot-cho-jung-lee} M. Ch\={o}, I. B. Jung, W. Y. Lee, On Aluthge transforms of p-hyponormal operators, Integr. Equ. Oper. Theory {\bf 53} (2005), pp. 321-329.
\bibitem{1966-pjm-embry} M. R. Embry, Conditions implying normality in Hilbert space, Pacific J. Math. {\bf 18} (1966), 457-460.
\bibitem{1970-pjm-embry} M. R. Embry, Similarities involving normal operators on Hilbert space, Pacific J. Math. {\bf 35} (1970), 331-336.
\bibitem{2003-pjm-foias-jung-ko-pearcy} C. Foia\c{s}, I. B. Jung, E. Ko, C. Pearcy, Complete contractivity of maps associated with the Aluthge and Duggal transforms, Pacific Journal of Math. {\bf 209} (2003), 249-259.
\bibitem{2004-ieot-i-y-y} M. Ito, T. Yamazaki, M. Yanagida, On the polar decomposition of the product of
two operators and its applications, Integr. Equ. Oper. Theory, {\bf 49} (2004), 461-472.
\bibitem{2004-jot-ito-yamazaki-yanagida} M. Ito, T. Yamazaki, M. Yanagida, On the polar decomposition of the Aluthge transformation and related results, J. Operator Theory {\bf 51} (2004), 303-319.
\bibitem{2012-mams-j-j-s} Z. J. Jab{\l}o\'nski,  I. B. Jung, J. Stochel, Weighted shifts on directed trees, {\em Mem. Amer. Math. Soc.} {\bf 216} (2012), no.\ 1017, viii+107pp.
\bibitem{1989-am-janas} J. Janas, On unbounded hyponormal operators, Ark. Mat. {\bf 27} (1989), 273-281.
\bibitem{2003-ieot-jung-ko-pearcy} I. Jung, E. Ko, C. Pearcy, The iterated Aluthge transform of an operator, Integr. Equ. Oper. Theory {\bf 45} (2003), 375-387.
\bibitem{2012-ieot-lee-lee-yoon} S. H. Lee, W. Y. Lee, J. Yoon, Subnormality of Aluthge Transforms of Weighted Shifts, Integr. Equ. Oper. Theory {\bf 72} (2012), 241-251.
\bibitem{1959-am-nelson} E. Nelson, Analytic vectors, Ann. Math. {\bf 70} (1959), 572-615.
\bibitem{1971-pams-olson} M. P. Olson, The selfadjoint operators of a von Neumann algebra form a conditionally complete lattice, Proc. Amer. Math. Soc. {\bf 28} (1971), 537-544.
\bibitem{1995-jfa-paulsen-pearcy-petrovic} V. Paulsen. C. Pearcy, S. Petrovic, On centered and weakly centered operators, J. Funct. Anal. {\bf 128} (1995), 87-101.
\bibitem{1992-jfa-petrovic} S. Petrović, A dilation theory for polynomially bounded operators, J. Funct. Anal. {\bf 108} (1992), 458-469.
\bibitem{2012-jmaa-planeta-stochel} A. Płaneta, J. Stochel, Spectral order for unbounded operators,  J. Math. Anal. Appl. {\bf 389} (2012), 1029-1045
\bibitem{2012-asm-pozzi} E. Pozzi, Universality of weighted composition operators on $L^2([0, 1])$ and Sobolev spaces, Acta Scientarum Math. {\bf 78} (2012), 609-642.
\bibitem{rud} W. Rudin, Real and Complex Analysis, McGraw-Hill, 1987.
\bibitem{schmudgen} K. Schmudgen, Unbounded Self-adjoint Operators on Hilbert Space, Graduate Texts in Mathematics {\bf 265} (2012), Springer.
\bibitem{sto-1990-hmj} J. Stochel, Seminormal composition operators on ${L}^2$ spaces induced by matrices, Hokkaido Math. J. {\bf 19} (1990), 307--324.
\bibitem{2024-laa-stochel-stochel} J. Stochel, J. B. Stochel, The hyperbolic cosine transform and its applications to composition operators, Linear Algebra Appl. {\bf 691} (2024), 1-36.
\bibitem{2021-lma-holleman-mcclatchy-thompson} D. Thompson, T. McClatchey, C. Holleman, Binormal, complex symmetric operators, Linear Multilinear Algebra {\bf 69} (2021), 1705-1715.
\bibitem{2015-jmaa-trepkowski} J. Trepkowski, Aluthge transforms of weighted shifts on directed trees, J. Math. Anal. Appl. {\bf 425} (2015), 886-899.
\end{thebibliography}

\end{document}